\documentclass[14pt, twoside, a4paper]{article} %
\usepackage{amsmath}

\usepackage{tikz}
\usepackage{graphicx}
\usepackage{subfigure}
\usepackage{calc}
\usepackage{float}

\usepackage[T1]{fontenc}
\usepackage[utf8]{inputenc}
\usepackage[english]{babel}

\let\getprepared\relax
\ifx\TestIngCommand\undefined\relax
\else\input ../../../proc/ppreamb-book.tex 
  \input init.tex \fi
\let\TestIngCommand\undefined
\newtheorem{example}{Example}

\newtheorem{conj}{Conjecture}
\usepackage{amsmath,amscd,amssymb}                                         
\usepackage{graphics}                                                     
\newtheorem{theo}{Theorem}                                                 
\newtheorem{lem}{Lemma}                                                    
\newtheorem{prop}{Proposition}                                             

\newtheorem{proof}{Proof}                                              
\newskip\ttglue\ttglue=.5em plus.25em minus.15em                           
\def\firstname#1{\def\FIRSTNAME{#1}\ignorespaces}
\def\lastname#1{\def\LASTNAME{#1}\ignorespaces}
\def\middleinitial#1{\def\MIDDLEINI{#1}\ignorespaces}
\def\department#1{\def\DEPARTMENT{#1}\ignorespaces}
\def\institute#1{\def\INSTITUTE{#1}\ignorespaces}
\def\address#1{\def\ADDRESS{#1}\ignorespaces}
\def\country#1{\def\COUNTRY{#1}\ignorespaces}
\def\otheraffiliation#1{\def\OTHERAFFILIATION{#1}\ignorespaces}
\def\email#1{\def\EMAIL{#1}\ignorespaces}
\newcount\autcount\autcount=0                                              
\newcount\affcount\affcount=0                                              
\newcount\numcount\newcount\nummcount                                      
\newcount\nummmcount\newcount\nummmmcount                                  
\def\writename#1#2{\ \kern-1ex\hbox{
  \csname AUthor\the#1\endcsname\                                          
  \edef\TESTSTR{}\expandafter\ifx\csname auTHor\the#1\endcsname\TESTSTR    
  \else\csname auTHor\the#1\endcsname.\ \fi                                
  \csname authOR\the#1\endcsname$^{\csname AFF\the#1\endcsname}$
  \expandafter\ifx\csname corr\number#1\endcsname\relax                    
  \else\thanks{Corresponding author.}\ \fi                                 
  }\ifnum#1<#2, \else\ \kern-1ex\fi}
\def\writeemail#1{
  \nummcount=0\relax\nummmcount=0\relax                                    
  \loop\ifnum\nummcount<\autcount\advance\nummcount by1\relax              
    {\expandafter\ifnum\csname AFF\the\nummcount\endcsname=#1\relax        
    \global\advance\nummmcount by1\fi}\repeat                              
  \nummcount=0\relax\nummmmcount=0\relax                                   
  \loop\ifnum\nummcount<\autcount\advance\nummcount by1\relax              
    {\expandafter\ifnum\csname AFF\the\nummcount\endcsname=#1\relax        
    \global\advance\nummmmcount by1\relax\def\blank{}\expandafter          
    \ifx\csname EMAIL\the\nummcount\endcsname\blank(no e-mail)
    \else\csname EMAIL\the\nummcount\endcsname                             
    \fi                                                                    
    \ifnum\nummmmcount<\nummmcount; \fi\fi}\repeat}
\long\def\BeginAuthorList#1\EndAuthorList{#1\relax                         
  \author{\vbox{\hsize=390pt\noindent\numcount=0\relax                     
    \loop\ifnum\numcount<\autcount\advance\numcount by1\relax              
      \writename{\numcount}{\autcount}
      \repeat}\\[2mm]                                                      
    \vbox{\small\numcount=0\relax                                          
      \loop\ifnum\numcount<\affcount\advance\numcount by1\relax            
        \vbox{{\count0=\numcount\relax                                     
          \loop\expandafter\ifnum\csname AFF\the\count0\endcsname
            <\numcount\relax\advance\count0 by1\relax\repeat               
          $^{\csname AFF\the\count0\endcsname}$}
        \def\BLANK{}\expandafter\ifx\csname DEPT\the\numcount\endcsname    
          \BLANK                                                           
          \else\csname DEPT\the\numcount\endcsname, \fi                    
        \csname INST\the\numcount\endcsname,                               
        \csname ADDR\the\numcount\endcsname,                               
        \csname COUN\the\numcount\endcsname                                
        \edef\TEST{}\expandafter\ifx\csname OTHE\the\numcount\endcsname
          \TEST                                                            
          .\else;\break\csname OTHE\the\numcount\endcsname.\fi}
        \vbox{\writeemail{\numcount}}
        \repeat}\\}}
\expandafter\def\csname x1\endcsname{}
\expandafter\def\csname x2\endcsname{}
\expandafter\def\csname x3\endcsname{}
\expandafter\def\csname x4\endcsname{}
\expandafter\def\csname x5\endcsname{}
\expandafter\def\csname x6\endcsname{}
\expandafter\def\csname x7\endcsname{}
\expandafter\def\csname x8\endcsname{}
\expandafter\def\csname x9\endcsname{}
\def\Author#1#2{\global\advance\autcount by1\relax#2                       
  \expandafter\edef\csname AUthor\the\autcount\endcsname{\FIRSTNAME}
  \expandafter\edef\csname auTHor\the\autcount\endcsname{\MIDDLEINI}
  \expandafter\edef\csname authOR\the\autcount\endcsname{\LASTNAME}
  \expandafter\edef\csname EMAIL\the\autcount\endcsname{\EMAIL}
  \let\tempera\"\def\"{\string\"}\expandafter\ifx\csname x\DEPARTMENT
    \endcsname\relax                                                       
    \global\advance\affcount by1\relax\let\"\tempera                       
    \expandafter\edef\csname DEPT\the\affcount\endcsname{\DEPARTMENT}
    \expandafter\edef\csname INST\the\affcount\endcsname{\INSTITUTE}
    \expandafter\edef\csname ADDR\the\affcount\endcsname{\ADDRESS}
    \expandafter\edef\csname COUN\the\affcount\endcsname{\COUNTRY}
    \expandafter\edef\csname OTHE\the\affcount\endcsname{\OTHERAFFILIATION}
    \expandafter\edef\csname AFF\the\autcount\endcsname{\the\affcount}
  \else\expandafter\edef\csname AFF\the\autcount\endcsname{\DEPARTMENT}
  \fi\let\"\tempera\ignorespaces}
\def\CorrespondingAuthor#1#2{
  \expandafter\xdef\csname corr\number#1\endcsname{cor}
  \Author#1{#2}}
\def\PaperTitle#1{\title{\bf#1}}
\def\Category#1{\ignorespaces}
\def\keywords#1{{\noindent \emph{Keywords:}                                
  \def\BLANK{}\def\TEST{#1}\ifx\BLANK\TEST(n/a).\else#1\fi}}
\getprepared                                                               
\begin{document}                                                           

\PaperTitle{Joint distribution of leftmost digits in positional notation and Schanuels's conjecture}
\Category{(Pure) Mathematics}

\date{}

\BeginAuthorList
  \Author1{
    \firstname{Wayne}
    \lastname{Lawton}
    \middleinitial{M}   
    \department{Department of the Theory of Functions, Institute of Mathematics and Computer Science}
    \institute{Siberian Federal University}
    \otheraffiliation{}
    \address{Krasnoyarsk}
    \country{Russian Federation}
    \email{wlawton50@gmail.com}}
\EndAuthorList
\maketitle
\thispagestyle{empty}
\begin{abstract}
\noindent Assume that $n \geq 2$ and $B = (b_1,...,b_n)$ has distince 
integer entries $\geq 3.$ For
$x > 0$ let $d_B(x) := (d_{b_1}(x),...,d_{b_n}(x))$ where 
$d_{b_i}(x) \in \{1,...,b_i-1\}$ is the leftmost digit in the 
base-$b_i$ positional notation representation of $x.$
We prove that if $d_B$ is surjective, then $\ln b_i$ and $\ln b_j$ are rationally independent whenever $i \neq j.$ We prove the converse for $n = 2,$ and for $n \geq 3$ if $\{\ln p : p \mbox{ prime} \}$ is algebraically independent, a condition implied by Schanuel's conjecture about transcendental numbers.
\end{abstract}
\footnote{\thanks{This work is supported by the Krasnoyarsk Mathematical Center and financed by the 
Ministry of Science and Higher Education of the Russian Federation (Agreement No. 075-02-2026-1314).}}
\section{Positional notation}\label{sec1}
$\mathbb Z, \mathbb Q, \mathbb R, \mathbb R_+$ denote integer, rational, real and positive numers, and iff mean if and only if. For integers $b \geq 3,$ $e \geq 1,$ 
$j \in \{1,...,b-1\}$ define 
$d_b : \mathbb R_+ \mapsto \{1,...,b-1\}$ and 
$S(b,e,j) \subset \{1,...,b^e - 1 \}$ by
\begin{equation}\label{eq1}
	d_b(x) = j \iff \log_b(x) \in [\log_b(j), \log_b(j+1)) + \mathbb Z, \ \ x \in \mathbb R_+,
\end{equation}
\begin{equation}\label{eq2}
	S(b,e,j) := 
\bigcup_{\ell = 0}^{e-1} 
\left( b^\ell j +\{0,...,b^\ell - 1\} \right).
\end{equation}
$d_b(x)$ is the leftmost 'digit' in the base-$b$ positional notation representation of $x.$
\begin{lem}\label{lem1}
$d_{b^e}(x)$ determines $d_b(x)$ by
\begin{equation}\label{eq3}
 d_b(x) = j \iff d_{b^e}(x) \in S(b,e,j).
\end{equation}
\end{lem}
\begin{proof}
The identity $\log_b(x) = e\, \log_{b^e}(x)$ implies
$$d_b(x) = j \iff  \log_{b^e}(x) \in [\log_{b^e}(j),\log_{b^e}(j+1)) + e^{-1}\mathbb Z$$
$$\iff \log_{b^e}(x) \in [\log_{b^e}(j),\log_{b^e}(j+1))+ e^{-1}\{0,...,e-1\} + \mathbb Z$$
$$\iff  \log_{b^e}(x) \in \bigcup_{\ell = 0}^{e-1} 
[\log_{b^e}(b^\ell j),\log_{b^e}(b^\ell (j+1))) + \mathbb Z$$
$$\iff  d_{b^e}(x) \in \bigcup_{\ell = 0}^{e-1} 
\mathbb Z \cap [b^\ell j,b^\ell (j+1)) \iff  d_{b^e}(x) \in S(b,e,j).$$
\end{proof}
\section{Joint distribution for two bases}\label{sec1}
Assume that $B = (b_1,b_2)$ has distinct integer entrues $\ge 3.$ Define 
$$d_B : \mathbb R_+ \mapsto \{1,...,b_1-1\} \times \{1,...,b_2-1\}$$ 
by
$d_B(x) := (d_{b_1}(x), d_{b_2}(x)).$
$\ln b_1$ and $\ln b_2$ are rationally dependent iff there exist integers $a, e_1, e_2$ such that $a \geq 2,$ $gcd(e_1,e_2) = 1,$ and $b_i = a^{e_i}, i = 1, 2.$  Then
$b := b_1^{e_2} = b_2^{e_1}$ and
Lemma \ref{lem1} implies 
\begin{equation}\label{eq4}
 d_B(x) = (j_1,j_2) \iff d_b(x) \in
 S(b_1,e_2,j_1) \cap S(b_2,e_1,j_2).
\end{equation}
\begin{example}\label{ex1}
If $B = (4,8),$ then $a = 2, e_1 = 2, e_2 = 2$ and $b = 64.$ We used 
(\ref{eq4}) to compute $(j_1,j_2) = d_B(x)$ as a function of $d_{64}(x) \in \{1,...,63\}$
displayed in the table below. 
$d_B$ is not surjective because its image excludes the six pairs  
$(2,3), (2,6), (2,7), (3,2), (3,4), (3,5).$ 
\end{example}
$$
\begin{array}{cccccccc}
      		 &  & j_1 = 1    	 & &  j_1 =2    & & j_1 = 3      \\
  j_2 = 1     &  & 1              & &  8-11   & &    12-15  \\
  j_2 = 2     &  &  16-23      & &  2        & &   \emptyset       \\ 
  j_2 = 3     &  &  24-32      & &   \emptyset        & &  3         \\
  j_2=  4     & 	&	4	           &	&  32-39   & &  \emptyset       \\
  j_2=  5     & 	&	5		      &	&  40-47    & &   \emptyset        \\
  j_2=  6     &  &  6             & &  \emptyset       & &  48-55    \\
  j_2 =  7    &	&	7	           &	&   \emptyset         & & 56-63
\end{array}
$$
Define the circle group 
$\mathbb T := \mathbb R / \mathbb Z$ and $n$-dimensional torus group
$\mathbb T^n = \mathbb R^n / \mathbb Z^n.$
For $\omega = (\omega_1,...,\omega_n) \in \mathbb R^n$ define the homomorphism
$\varphi_{\omega} : \mathbb R \mapsto \mathbb T^n$ by 
$\varphi_\omega(t) := t\omega + \mathbb Z^n.$
\begin{lem}\label{lem2}
The image of $\varphi_\omega$ is dense in 
$\mathbb T^n$ iff $\omega_1,...,\omega_n$ are rationally indpendent.
\end{lem}
\begin{proof}
This landmark result was proved in 1884 by Leopold Kronecker \cite{kronecker}. A modern English langauge proof is given in
Proposition 1.5.1 in \cite{katok}.
\end{proof}
\begin{theo}\label{thm1}
If $\ln b_1$ and $\ln b_2$ are rationally dependent and integers $a, e_1, e_2$ satisfy $a \geq 2,$ 
$gcd(e_1,e_2) = 1,$ and $b_i = a^{e_i} \geq 3, i = 1, 2,$ 
then 
$(j_1,j_2) \in \{1,...,b_1-1\} \times \{1,...,b_2-1\}$
is in the image of
$d_B$ iff there exists an integer $c$ such that
\begin{equation}\label{eq5}
\frac{j_1}{j_2+1} < a^c < 
\frac{j_1+1}{j_2}.
\end{equation}
Since $a^c$ never belongs to  the open interval 
$(\frac{1}{2},1),$ $d_B$ is not surjective since $(2,3)$ is 
not contained in its image.
If $\ln b_1$ and $\ln b_2$ are rationally independent, then $d_B$ is surjective.
\end{theo}
\begin{proof} If $\ln b_1$ and $\ln b_2$ are rationally dependent 
and $x \in \mathbb R_+$ then $d_B(x) = (j_1,j_2)$ iff 
there exist integers $d_1$ and $d_2$ such that
\begin{equation}\label{eq6}
x \in a^{e_1d_1}[j_1, j_1+1) \cap a^{e_2d_2}[j_2, j_2+1).
\end{equation}
Then 
(\ref{eq5}) holds for
$c = e_2d_2 - e_1d_1.$ Conversly, if
(\ref{eq5}) holds then since $gcd(e_1,e_2) = 1,$ there exist integers $d_1$ and $d_2$ such that $c = e_2d_2 - e_1d_1$ and then (\ref{eq6}) holds.
If $\ln b_1$ and $\ln b_2$ are rationally independent. Then 
$\omega_1 := \frac{1}{\ln b_1}$ and 
$\omega_2 := \frac{1}{\ln b_2}$
are rationally independent so Lemma \ref{lem2} 
and the surjectivtivity of $\ln : \mathbb R_+ \mapsto \mathbb R$ implies that the image of the composition 
$h := \varphi_\omega \circ \ln : \mathbb R_+ \mapsto \mathbb T^2$ is dense in $\mathbb T^2.$
Partition $\mathbb T^2$ into sets 
$$R(j_1,j_2) := 
[\log_{b_1}(j_1), \log_{b_1}(j_1+1)) \times 
[\log_{b_2}(j_2), \log_{b_2}(j_2+1)) + \mathbb Z^2,
\}.$$ 
Then $d_B(x) = (j_1,j_2) \iff h(x) \in R(j_1,j_2)$ and
the surjectivity of $d_B$ follows since $h^\mathbb R_+)$ 
us dens in $\mathbb T^2$ and each $R(\j_1,j_2)$ contains
a nonempty open subset.
\end{proof}
\section{Schanuel's conjecture and its consequences}\label{sec3}
Schanuel's conjecture (\cite{lang}, p. 30--32) asserts:
If $c_1,...,c_m$ are rationaly independent complex numbers, 
then the field
$\mathbb Q(c_1,...,c_m, exp(c_1),...,exp(c_m))$ 
has transcendence degree at least $m$ over 
$\mathbb Q.$ 
\newline
\\
Let $\mathbb Q(z_1,...,z_m)$ denote the field of rational functions in variables $z_1,...,z_m$ represented by the field of meromorphic functions on the
$n$-fold product of Riemann spheres $(\mathbb C \cup \{\infty\})^m.$
Assume that $m \geq 1$ and $p_1,...,p_m$ are distinct primes,
\begin{conj}\label{conj1}
 $\ln p_1,...,\ln p_m$ are algebraically independent.
Equivalently, the homomorphism 
$\psi : \mathbb Q(z_1,...,z_m) \mapsto \mathbb Q(\ln p_1,...,\ln p_m)$ 
satisfying $\psi(z_j) = \ln p_j$ is an isomorphism.
\end{conj}
\begin{lem}\label{lem3}
Schanuel's conjecture implies conjecture \ref{conj1} 
\end{lem}
\begin{proof}
The  unique factorization of positive integers 
as products of power of primes implies that 
$\ln p_1,...,\ln p_m$ are rationally independent.
Conjecture \ref{conj1} then follows form
Schanuel's conjecture by
defining $c_i := \ln p_i, i = 1,...,n.$
\end{proof}
\begin{conj}\label{conj2}
If $b_1,...,b_n$ are integers $\geq 3$
such that no pair of numbers in the set 
$\{\ln b_1,...,\ln b_n\}$ 
are rationaly dependent, then 
$\frac{1}{\ln b_1},...,\frac{1}{\ln b_n}$ 
are rationally independent.
\end{conj}
\begin{prop}\label{prop1}
Conjecture \ref{conj1} implies conjecture \ref{conj2}. 
\end{prop}
\begin{proof}
Let $p_1,...,p_m$ be the distinct primes that divide $b_1 \cdots b_n.$
Then there exist integers $e_{i,j} \geq 0, i = 1,...,n; j = 1,...,m$ such that
\begin{equation}\label{eq7}
 b_i = p_1^{e_{i,1}} \cdots p_m^{e_{i,m}}, \ \ i = 1,...,n.
\end{equation}
Conjecture \ref{conj1} implies that $\psi$ is an isomorphism hence 
$\psi^{-1}(\ln b_i) = \sum_{j = 1}^m e_{i,j} z_j,$ so
\begin{equation}\label{eq8}
 \zeta_i := \psi^{-1}\left(\frac{1}{\ln b_i}\right) = 
\left( \sum_{j = 1}^m e_{i,j} z_j \right)^{-1}, \ \ i = 1,...,n.
\end{equation}
$\zeta_i(z_1,...,z_m) = \infty$ on an
$(m-1)$-dimensional $\mathbb C$-subspaces $S_i$ of 
$\mathbb C^m \subset (\mathbb C \cup \{\infty\})^m.$
The hypothesis in conjecture \ref{conj2} implies that 
$S_1,...,S_m$ are distinct. Therefore 
$\zeta_1,...,\zeta_m$
and hence
$\frac{1}{\ln b_1},..., \frac{1}{\ln b_1}$
are rationally independent.
\end{proof}
\section{Joint distribution for $n \geq 3$ bases}\label{sec4}
Letf $B = (b_1,...,b_n)$ with distinct integer entries $\geq 3$ 
and defin $\omega := (\omega_1,...,\omega_n), \omega_i := \frac{1}{\ln b_i}.$ Then 
$
d_B : \mathbb R_+ \mapsto \{1,...,b_1-1\} \times \cdots \times \{1,...,b_n-1\}.
$
\begin{theo}\label{thm2}
If some pair of numbers in the set $\{\ln b_1,...,\ln b_n\}$ 
are rationaly dependent, then  $d_B$ is not surjective. 
If no pair of numbers in the set $\{\ln b_1,...,\ln b_n\}$ 
are rationaly dependent, then 
conjecture \ref{conj2}, and hence  Schanuel's conjecture, implies $d_B$ is surjective.
\end{theo}
\begin{proof}
The first assertion follows from Theorem \ref{thm1}. Assume that no pair of numbers in the set 
$\{\ln b_1,...,\ln b_n\}$ 
are rationaly dependent and assume conjecture \ref{conj2}. 
Then the entries of 
$\omega$ are rationally independent. Therefore 
Lemma \ref{lem2} implies that the image of
$h := \psi_\omega \circ \ln : \mathbb R_+ 
\mapsto \mathbb T^n$ is dense in 
$\mathbb T^n.$ 
The same argument used in the proof
of Theorem \ref{thm1} shows that
$d_B$ is surjective.
\end{proof}

\end{document}